\documentclass[reqno]{amsart}

\usepackage{mathrsfs,enumerate,amsmath,amssymb}

\usepackage{color}

\usepackage{hyperref}

\newtheorem{theorem}{Theorem}[section]
\newtheorem{remark}[theorem]{Remark}
\newtheorem{corollary}[theorem]{Corollary}
\newtheorem{lemma}[theorem]{Lemma}
\newtheorem{proposition}[theorem]{Proposition}
\newtheorem{definition}[theorem]{Definition}

\newcommand{\R}{\mathbb R}
\newcommand{\Om}{\Omega}
\newcommand{\haus}{\mathcal{H}^1}
\newcommand{\clen}{L^*}
\newcommand{\mi}{\mathcal{SM}}
\newcommand{\per}{\mathcal S}
\newcommand{\leb}{\mathcal{L}}
\newcommand{\la}{\lambda}
\newcommand{\K}{K_\text{\tiny opt}}
\newcommand{\hull}{\mathrm{hull}}

\title[Optimizing the first Dirichlet eigenvalue]{Optimizing the first Dirichlet eigenvalue of \\the Laplacian with an obstacle}

\author{Antoine Henrot}
\address[A.\@ Henrot]{Institut \'Elie Cartan de Lorraine, UMR 7502 Universit\'e de Lorraine and CNRS, B.P.\@ 70239, 54506 Vandoeuvre-l\`es-Nancy, France}
\email{antoine.henrot@univ-lorraine.fr}
\author{Davide Zucco}
\address[D.\@ Zucco]{Dipartimento di Matematica ``Giuseppe Peano'', Universit\`a di Torino, 10123 Torino, Italy}
\email{davide.zucco@unito.it}

\begin{document}

\begin{abstract}
Inside a fixed bounded domain $\Omega$ of the plane, we look for the best compact connected set $K$, of given
perimeter, in order to {\it maximize} the first Dirichlet eigenvalue $\lambda_1(\Omega\setminus K)$.
We discuss some of the qualitative properties of the maximizers, moving toward existence, regularity and geometry.
Then we study the problem in specific domains: disks, rings, and, more generally, disks with convex holes. In these situations, we prove symmetry and in some cases non symmetry results, identifying the solution.

We choose to work with the \emph{outer Minkowski content} as the ``good" notion of perimeter. Therefore, we are led to prove some
new properties for it as its lower semicontinuity with respect to the Hausdorff convergence and the fact
that the outer Minkowski content is equal to the Hausdorff lower semicontinuous envelope of the classical
perimeter.
\end{abstract}

\maketitle

\paragraph{\small \textbf{Key words.} First Laplace Eigenvalue, Shape Optimization, Outer Minkowski Content.}
 
\paragraph{\small \textbf{AMS subject classifications.} 28A75, 35P15, 49Q10, 49Q15.}

\section{Introduction}

A shape optimization problem is composed by three prime ingredients: the cost functional, the class of admissible sets and the constraints. One also has to decide whether to tackle the minimization problem or the maximization problem. Often, only one of
these is relevant.

A well studied cost functional is the one which models the principal frequency of vibration of a membrane. A shape optimization problem for this quantity has a long history going back at least to 1877 when Lord Rayleigh observed and conjectured that among all membranes of given area, the disk has the minimum principal frequency. 
This conjecture was solved by Faber in 1923 and, independently, by Krahn in 1924, and many other contributions to similar problems for the principal frequency appeared during the 20th and 21st centuries, see \cite{hen}, \cite{H2}. 
In 1963 Hersch, using the work of Payne and Weinberger \cite{paywei}, proved that of all doubly connected membranes of given area and perimeter, the ring has the maximal principal frequency (see \cite{her1, her2} and also Theorem~\ref{hpw} of this paper for the precise statement).

The principal frequency of a membrane fixed along its boundary is mathematically described by the first eigenvalue of the Laplacian $\la_1$ of a bounded domain $\Om$ in the plane with Dirichlet boundary conditions on $\partial \Om$. In this paper we discuss various shape optimization problems in the plane for the first eigenvalue of the Laplacian with Dirichlet boundary conditions on a domain with an \emph{obstacle}: inside a given bounded domain $\Om\subset \R^2$ our goal is to find the best compact set $K$ (the obstacle) so as to optimize the first Dirichlet eigenvalue of the open set $\Om\setminus K$, namely
\[
\mathrm{opt} \{ \la_1(\Om\setminus K) : \; \text{$K\subset\overline{\Om}$, $K$ closed and subjects to additional constraints}\}.
\]
We recall the variational characterization of $\la_1(\Om\setminus K)$ as the minimum of the Rayleigh quotient among all non zero functions in the Sobolev space $H^1_0(\Omega\setminus K)$, that is
\[
\la_1(\Omega\setminus K):=\min_{\substack{u\in H_0^1(\Omega\setminus K)\\ u\neq 0}} \frac{\int_{\Omega\setminus K} |\nabla u(x)|^2\,dx}{\int_{\Omega\setminus K} u(x)^2\, dx}.
\]
Essentially, we are imposing the Dirichlet condition over $\partial \Om$ and over a \emph{supplementary} region $K$, of possibly positive measure, and we look for the best {obstacle}, both in shape and location, which optimizes this eigenvalue. 
Similar problems in this spirit were tackled by Ramm and Shivakumar in \cite{ram-shi}, Harrel, Kr\"oger and Kurata in \cite{hakrku}, Kesavan in \cite{Kes}, and by El Soufi and Kiwan in \cite{esokiw}. However, in these papers the authors considered obstacles of \emph{fixed} shape, trying to optimize the first eigenvalue only by means of \emph{rigid motions}, translating or rotating the obstacle. Of course, the results that can be obtained dealing with a wider class of admissible obstacles, whose shapes can vary,  are weaker than those obtained in the
previous papers and it is hopeless to expect a universal solution (in general the shape and the location of an optimal obstacle depends on the domain $\Om$). More recently, Tilli and Zucco have treated in \cite{tilzuc, tilzuc2} a related problem for the first eigenvalue of an elliptic operator in divergence form, in the restricted class of one-dimensional obstacles (i.e., sets of finite one-dimensional Hausdorff measure). 

{As additional constraints, in the optimization problem above, we are interested in \emph{area} or \emph{perimeter}.  
For the area we use the Lebesgue measure, while for the perimeter we must choose it careful.}
Indeed, since objects of positive capacity but of
Lebesgue measure zero influence the first eigenvalue but are not seen by the {classical perimeter} (as defined by De Giorgi),
we need to choose another notion of perimeter, more sensitive to one-dimensional objects. We want to work with 
a perimeter that
\begin{itemize}
\item[(a)]  coincides with the classical notion of perimeter on \emph{regular} sets;
\item[(b)] measures twice the length of \emph{one-dimensional} objects;
\item[(c)] is \emph{lower semicontinuous} with respect to some convergence.
\end{itemize}
Item (a) is natural and allows to use results available on the classical perimeter. Item (b) rules out thin competitors (for instance it penalizes segments), helping for the identification of solutions in specific domains. Item (c) is crucial for the existence in the optimization problems with perimeter constraint. 
Therefore, we are led to consider the following quantity (see \cite{amcovi}). 

\begin{definition}[Outer Minkowski content]\label{minkowski}
Let $K\subset \mathbb R^2$ be a closed set. The \emph{upper} and \emph{lower outer Minkowski contents} $\mi_*(K)$ and $\mi^*(K)$ of $K$ are defined, respectively, as 
\begin{equation*}
\mi_*(K):=\liminf_{\epsilon\to 0}\frac{\leb(K^{\oplus\epsilon}\setminus K)}{\epsilon} \quad \text{and}\quad \mi^*(K):=\limsup_{\epsilon\to 0}\frac{\leb(K^{\oplus\epsilon}\setminus K)}{\epsilon},
\end{equation*}
where $K^{\oplus\epsilon}:=\{x\in \R^2:\, \mathrm{d}_K(x)\leq \epsilon\}$ is the $\epsilon$-tubular neighborhood of $K$ through the distance function $\mathrm{d}_K$ to $K$.
If $\mi_*(K) = \mi^*(K)$ their common value is denoted by $\mi(K)$ and called \emph{outer Minkowski content} of $K$.
\end{definition} 
In Section~\ref{sec.2} we prove that the outer Minkowski content of any compact, \emph{connected} and non empty set of the plane always exists and that satisfies the three items listed above. Notice that, while items (a) and (b) are well-known to be satisfied by the outer Minkowski content (see \cite{amcovi}), this is not the same for the third one (we have not found any proof in the literature). 
More specifically, we prove the lower semicontinuity of the outer Minkowski content with respect to the Hausdorff convergence. Moreover, we use this result to answer a question posed by Cerf in \cite{cerf}: we show that the outer Minkowski content is equal to the \emph{Hausdorff lower semicontinuous envelope of the classical perimeter}.

We are now ready to formulate four relevant problems for the first Dirichlet eigenvalue on a bounded domain $\Om\subset\R^2$, depending on the type of optimization (\emph{minimization} or \emph{maximization}) and the kind of constraint (\emph{area} or \emph{perimeter}); the class of admissible obstacles will be chosen accordingly to guarantee existence of solutions. Then, in the rest of the paper, we will focus on the last of these four problems. Notice that we do not require any regularity of the boundary $\partial \Om$.

\subsection*{{Problem 1: minimizing the first eigenvalue with an obstacle of given area}.}

For a fixed $A\in (0, \leb(\Om))$, with $\leb$ the Lebesgue measure, consider the \emph{minimization} problem
\begin{equation}\label{prob1}
\min \{ \la_1(\Om\setminus K) : \; 	\text{$K\subset \overline{\Om}$, $K$ closed,   $\leb(K)=A$}\}.
\end{equation}
This problem is related to the minimization of the first eigenvalue among \emph{open} sets constrained to lie in a given {box} (and also with a given area), see \cite[Section 3.4]{hen}. Indeed, passing to the complementary set $O=\Om\setminus K$ problem \eqref{prob1} becomes equivalent to the minimization of $\la_1(O)$ among open sets $O\subseteq \Om$ of area $\leb({\overline{\Om}})-A$ (in this framework $\Om$ represents the box). 
Therefore, from what is known on the minimizers contained into a box, we infer the existence of a solution of \eqref{prob1} and some of its qualitative properties. 
We have to distinguish two cases, depending on the existence of disks of area $\leb(\overline{\Om})-A$ that are contained inside $ \Om$ (to this aim we introduce the inradius  $\rho(\Omega)$ of $\Om$).
\begin{itemize}
\item[-] Let $A\geq \leb(\overline{\Om})-\pi \rho(\Om)^2$. Thanks to the Faber-Krahn inequality, a closed set $\K$ minimizes \eqref{prob1} if and only if $\K=\overline{\Om}\setminus B$ with the \emph{open} set $B$ that is (up to sets of capacity zero) any disk in $\Om$ of area $\leb(\overline{\Om})-A$.
This imply that, in general, problem \eqref{prob1} does not have a unique solution.
\item[-] Let $A< \leb(\overline{\Om})-\pi \rho(\Om)^2$. By \cite[Theorem 3.4.1]{hen}, \cite{bri-lam} and \cite{hen-oud}, every minimizer $\K$ of \eqref{prob1} touches the boundary of $\Om$, its free boundary (i.e., the part of the boundary of $\K$ which is inside $\Omega$) is analytic and does not contain any arc of circle.
\end{itemize}

\subsection*{Problem 2: maximizing the first eigenvalue with an obstacle of given~area.}

The corresponding \emph{maximization} problem of \eqref{prob1} has no solutions. Indeed, one can construct a sequence of closed sets $K_n\subset\overline\Om$ of Lebesgue measure $A$ so that  $\la_1(\Om\setminus K_n)\uparrow \infty$ as $n\to\infty$ (for instance by taking $K_n$ as the union of a \emph{given} closed set in $\overline \Om$ of area $A$ with a curve filling $\overline \Omega$ as $n$ increases, see \cite{tilzuc, tilzuc2} where the limit distribution in $\overline{\Om}$ of such curves is studied in detail). 
To guarantee the existence of a maximizer one needs to prevent maximizing sequences to spread out over $\overline\Om$. This can be achieved by imposing stronger geometrical constraints on the class of admissible obstacles (notice that connectedness is still not sufficient). 
Therefore, for a fixed  $A\in (0, \leb(\Om))$, we are led to consider the maximization problem
\begin{equation}\label{prob2}
\max \{ \la_1(\Om\setminus K) : \;	 \text{$K\subset \overline{\Om}$, $K$ closed and convex,   $\leb(K)=A$}\}.
\end{equation}
Now, the existence of a maximizer in the restricted class of convex sets is straightforward (see \cite{bucbut,hen}). Moreover, as convexity seems necessary for the existence, it is natural to expect every solution of \eqref{prob2} to \emph{saturate} 
the convexity constraint, in the sense that the boundary of any solution should contain non-strictly convex parts. 
In particular, it would be interesting to know whether this maximization problem has only polygonal sets as solutions, see \cite{lam-nov}, \cite{lam-nov-pie} for results in this direction for shape optimization problems with convexity constraints.

\subsection*{{Problem 3: minimizing the first eigenvalue with an obstacle of given perimeter}.}
The corresponding minimization problem of \eqref{prob1} with the area constraint replaced by a \emph{perimeter} constraint (whatever notion of perimeter one would consider) is in general not well-posed. Indeed, for every $L>0$, one can construct a sequence of smooth connected and closed sets $K_n\subset\overline\Om$ of perimeter $L$ approaching a subset of $\partial \Om$ so that $\la_1(\Om\setminus K_n)\downarrow \la_1(\Om)$ as $n\to\infty$ (notice that by regularity there is no doubt on the notion of perimeter of $K_n$).  
Therefore, as in Problem 2,  we restrict the class of admissible obstacles to convex sets. 
For a fixed $L\in (0,\mi(\overline{\Om}))$ (the existence of $\mi(\overline{\Om})$ will be provided by Lemma~\ref{Mchar}),  consider the {minimization} problem
\begin{equation}\label{prob3}
\min \{ \la_1(\Om\setminus K) : \; 	\text{$K\subset \overline{\Om}$, $K$ closed and convex,   $\mi(K)=L$}\}.
\end{equation}
The existence of a minimizer is a consequence of the compactness of the class of convex sets and of the continuity of $\mi$ w.r.t Hausdorff convergence of convex sets  (see \cite{bucbut,hen} and recall \eqref{regular} with \eqref{curve} of this paper). 
Notice that, for particular domains $\Om$ and small values $L$, it is still possible to have trivial solutions. 
For example, if the boundary $\partial\Om$ contains a segment and if $L$ is smaller
than twice the length of such a segment, then every segment contained in $\partial \Om$ of outer Minkowski content $L$ minimizes \eqref{prob3}.
On the other hand, if $L$ is large enough,
every minimizer has positive Lebesgue measure, since minimizing sequences will not be able to degenerate to a segment. In any case, one expects that every minimizer of \eqref{prob3} touches the boundary $\partial \Om$.

\subsection*{{Problem 4: maximizing the first eigenvalue with an obstacle of given perimeter.}} 
This is the problem that we analyze in detail in this paper. For a fixed $L\in(0,\mi(\overline \Om))$ (the existence of $\mi(\overline{\Om})$ will be provided by Lemma~\ref{Mchar}) we study the \emph{maximization} problem
\begin{equation}\label{problem}
\max\{\la_1(\Om\setminus K):\: \text{$K\subseteq \overline{\Om}$, $K$ continuum, $\mi(K)\leq L$}\},
\end{equation}
where, as usual, the word \emph{continuum} (continua for the plural) stands for a compact, connected and non empty set.
Notice that the class of admissible obstacles is very wide: a generic obstacle can be split into two pieces, a part of positive Lebesgue measure (\emph{the body}) and a part of Lebesgue measure zero (\emph{the tentacles}). Here connectedness of the admissible obstacles combined with the perimeter constraint, prevents maximizing sequences to spread out over $\overline\Om$ and it is sufficient for the existence of a solution (compare with Problem 2). Moreover, the assumption $L<\mi(\overline \Om)$, prevents to have non trivial solutions, otherwise by \eqref{regular} one could take $K=\overline{\Omega}$ and $\la_1(\emptyset)=\infty$. Notice that we have chosen the inequality in the perimeter constraint, {even though in the case of equality existence holds as well. }

In Section~\ref{sec.3}, for general bounded domains $\Om$, we discuss some of the qualitative properties of the maximizers of \eqref{problem}, moving toward existence, regularity and geometry.
Clearly any solution of \eqref{problem} depends on the geometry of the domain $\Omega$ and this is the reason why it is hard to find explicit solutions.  However, when the domain $\Omega$ has a specific shape, such as a disk, a ring, or more generally, a disk with convex holes, we are able to go beyond qualitative results. In Section~\ref{sec.4} we prove symmetry and, in some cases non symmetry results, identifying the solution for certain values of the constraint (actually for all values when the domain is itself a disk).
It is worth mentioning that when $\Om$ is a ring appears the not so common phenomenon of \emph{symmetry breaking}: for certain values of the constraint every maximizer is not radially symmetric.

\subsection*{Aknowledgements} 
The authors want to warmly thank the anonymous referee whose great work allows to significantly improve the
preliminary version of this paper.
The work of Antoine Henrot is supported by the project ANR-12-BS01-0007-01-OPTIFORM {\it Optimisation de formes}
financed by the French Agence Nationale de la Recherche (ANR). The work of Davide Zucco is  supported by the project 2015 
%\href{http://fcm2.weebly.com/}
{\it Fenomeni Critici nella Meccanica dei Materiali: un Approccio Variazionale} financed by the INdAM-GNAMPA.
This work has been developed at the Institut \'Elie Cartan de Lorraine and at the Scuola Internazionale Superiore di Studi Avanzati di Trieste: these institutions are kindly acknowledged for their warm hospitality.

\section{Some new properties of the outer Minkowski content}\label{sec.2}

In this section we prove some new properties of the outer Minkowski content (see Definition~\ref{minkowski}) on \emph{continua} (i.e., compact, connected, non empty sets) of the plane, such as its lower semicontinuity with respect to the Hausdorff convergence and the fact that it is equal to the Hausdorff lower semicontinuous envelope of the classical perimeter.

We start by fixing the notation and recalling some preliminary facts. Given a set $U\subset \R^2$ we denote by $\mathrm{int}(U)$, $\partial U$ and $\overline U$, respectively, the {interior}, the boundary and the closure of $U$. We will use $K_n\to^{H} K$ to denote a sequence $\{K_n\}$ of closed sets converging with respect to the Hausdorff convergence to a closed set $K$ as $n\to \infty$. Recall that the Hausdorff convergence preserves both connectedness and convexity (see \cite{ambtil, henpie} for other facts about the Hausdorff convergence). A \emph{residual domain} of $K$ in a set $U\subset \mathbb R^2$ is a connected component of $U \setminus K$. We recall that every residual domain of a continuum in $\mathbb R^2$ is simply connected and has a connected boundary (see \cite{cerf}). This implies that if a continuum $K$ has $k$ residual domains (with $k\in\mathbb N$) then $\partial K$ has at most $k$ connected components.

Now we recall the link of the outer Minkowski content of \emph{regular} continua with more classical quantities, such as the De Giorgi perimeter $\mathcal P$, the Hausdorff measure $\haus$ and the Minkowski content $\mathcal M$ (defined as its outer counterpart in Definition~\ref{minkowski} but computing the limits as $\epsilon\to 0$ of the ratio $\leb(K^{\oplus\epsilon})/(2\epsilon)$). 
Let $K\subset \R^2$ be compact. 
\begin{itemize}
\item[(a)] If $K$ is \emph{Lipschitz} (i.e., $K$ is locally the subgraph of a Lipschitz function near every boundary point of $K$) then
\begin{equation}\label{regular}
\mi(K)=\mathcal P(K)=\haus(\partial K)<+\infty.
\end{equation}
\item[(b)] If $K$ is \emph{$1$-rectifiable} 
(i.e., $K$ is the image of a compact subset of the real line through a Lipschitz function from $\R$ to $\R^2$) then
\begin{equation}\label{curve}
\mi(K)=2\mathcal{M}(K)=2\haus(\partial K)<+\infty.
\end{equation}
\end{itemize}
Item (a) has been proved in \cite[Corollary~1]{amcovi} (see also \cite{amfupa}). 
The first equality in item (b) is trivial since $\mi$ and $2\mathcal M$, by definition, differs only on sets of positive Lebesgue measure; the second equality has been proved in \cite[p. 275]{fed}. 

The lower semicontinuity of the outer Minkowski content with respect to the Hausdorff convergence is more tricky. We dedicate the rest of this section for its proof with some of its consequences, since it is interesting on its own and, to the best of our knowledge, new. 
To prove this result we adapt some ideas developed in \cite{buczol2} that were tailored for the so-called density perimeter. 
 
\begin{lemma}\label{approximate}
Let $K\subset \mathbb R^2$ be a continuum. Then there exists a sequence of continua $\{K_n\}$ such that $K_n$ is a finite union of segments and $K_n\to^H K$. 
\end{lemma}
\begin{proof}
We recall the construction given in the proof of \cite[Theorem~4.1]{buczol2}). Take a covering of $K$ given by open disks of radius $1/n$. By compactness the covering can be provided by a finite number of disks. Moreover, by connectedness this family of open disks can be considered connected. Then define the continuum $K_n$ as the family of all the segments connecting any two centers of those disks of this covering with non-empty intersection. By definition, $K_n$ satisfied what claimed.
\end{proof}

We will need the following characterization of the outer Minkowski content.

\begin{lemma}\label{Mchar}
Let $K\subset \R^2$ be a continuum. Then the outer Minkowski content of $K$ exists and
\begin{equation}\label{minkowski2}
\mi(K)=\sup_{\epsilon>0} \bigg[\frac{\leb(K^{\oplus\epsilon}\setminus K)}{\epsilon}-\pi\epsilon\bigg].
\end{equation}
\end{lemma}
\begin{proof}
To prove \eqref{minkowski2} it is sufficient to show that the quantity inside the $\sup$ is non-increasing with respect to $\epsilon$, namely that for every $0<\epsilon<\delta$ 
\begin{equation}\label{claim0}
\frac{\leb( K^{\oplus\delta}\setminus K)}{\delta}-\pi\delta\leq \frac{\leb({K^{\oplus\epsilon}}\setminus K)}{\epsilon}-\pi\epsilon.
\end{equation}
Indeed in this case the sup in \eqref{minkowski2} is a limit and guarantees the existence of the limit in the definition of $\mi$.
We prove \eqref{claim0} by an approximation argument. Let $\epsilon>0$ and $\delta>0$ be fixed. By Lemma~\ref{approximate} there exists a sequence of continua $\{K_n\}$ such that $K_n$  is a finite union of segments and $K_n\to^H K$. Since both $\leb(K_n)=0$ and $\leb(\partial K_n^{\oplus \epsilon})=0$ from the inequality (12) in \cite{buczol2} for every set $K_n$ of the approximating sequence we have
\[
\frac{\leb(K_n^{\oplus\delta} \setminus K_n)}{\delta} -\pi\delta \leq \frac{\leb({K_n^{\oplus\epsilon}} \setminus K_n)}{\epsilon}-\pi\epsilon.
\]
Therefore, by the continuity of the measure on increasing sequences of sets, 
for a fixed $\eta>0$ there exists $\mu\in(0,\epsilon)$ so that
\begin{equation}\label{proof1}
\frac{\leb(K_n^{\oplus\delta} \setminus K_n^{\oplus \mu})}{\delta} -\pi\delta \leq \frac{\leb({K_n^{\oplus\epsilon}} \setminus K_n^{\oplus \mu})}{\epsilon} -\pi\epsilon+\eta.\end{equation}
Now, by definition of the Hausdorff convergence, there exists $n_\mu$ such that 
$K_n\subset K^{\oplus \mu}$ and $K\subset K_n^{\oplus \mu}$ for every $n>n_\mu$. Then $K_n^{\oplus\epsilon}\subset K^{\oplus(\mu+\epsilon)}$, and \eqref{proof1} becomes
\begin{equation}\label{proof2}
\frac{\leb(K_n^{\oplus\delta} \setminus K_n^{\oplus \mu})}{\delta} -\pi\delta\leq \frac{\leb(K^{\oplus(\mu+\epsilon)} \setminus K)}{\epsilon} -\pi\epsilon+\eta.
\end{equation}
Using again the definition of the Hausdorff convergence, for all $\xi$ and $\theta$ with $0<\mu<\xi<\theta<\delta$, there exists $n_{\xi,\theta}$ such that $K^{\oplus \theta}\subset K_n^{\oplus\delta}$ and $K_n^{\oplus\mu}\subset K^{\oplus\xi}$ for every $n>n_{\xi,\theta}$. Then for $n>\max\{n_\mu,n_{\xi,\theta}\}$ the inequality \eqref{proof2} becomes
\[
\frac{\leb(K^{\oplus \theta} \setminus K^{\oplus\xi})}{\delta} -\pi\delta\leq \frac{\leb(K^{\oplus(\mu+\epsilon)} \setminus K)}{\epsilon} -\pi\epsilon+\eta.
\]
Since $$\bigcup_{\mu>0}\bigcup_{\substack{\xi, \theta \\ \mu<\xi<\theta<\delta}} K^{\oplus\theta}\setminus K^{\oplus\xi}=\textrm{int}(K^{\oplus\delta})\setminus K\quad \text{and} \quad \bigcap_{\mu>0} K^{\oplus(\mu+\epsilon)}\setminus K={K^{\oplus\epsilon}}\setminus K$$ the continuity of the measure on monotone (in the sense of set inclusion) sequences of sets with the fact that $\leb(\partial K^{\oplus\delta})=0$ gives
\[
\frac{\leb(K^{\oplus \delta} \setminus K)}{\delta} -\pi\delta\leq \frac{\leb({K^{\oplus \epsilon}} \setminus K)}{\epsilon} -\pi\epsilon+\eta.
\]
By the arbitrariness of $\eta$ we obtain  \eqref{claim0}. This gives \eqref{minkowski2} and then it follows the existence of the limits in Definition \ref{minkowski}.
\end{proof}

The lower semicontinuity of the outer Minkowski content with respect to the Hausdorff convergence is an immediate consequence of Lemma~\ref{Mchar}.

\begin{theorem}\label{Mlsc}
Let $\{K_n\}$ be a sequence of continua in $\R^2$ such that $K_n\to^H K$. Then $K\subset \mathbb R^2$ is a continuum and
\[
\mi(K)\leq \liminf_{n\to\infty} \mi(K_n).
\]
\end{theorem}
\begin{proof}
By \cite[Proposition~2.2.17]{henpie} the Hausdorff convergence preserves connectedness.
Let $\epsilon>0$ be fixed. By definition of the Hausdorff convergence, for every $\delta\in(0,\epsilon)$ there exists $n_\delta$ such that $K_n\subset K^{\oplus\delta}$ and $K\subset K_n^{\oplus\delta}$ for every $n>n_\delta$. This with \eqref{minkowski2} implies that
\[
\begin{split}
\frac{\leb(K^{\oplus\epsilon}\setminus K^{\oplus\delta})}{\epsilon}-\pi\epsilon&\leq\bigg( \frac{\leb{(K_n^{\oplus(\epsilon+\delta)}\setminus K_n)}}{\epsilon+\delta}-\pi(\epsilon+\delta)\bigg)\frac{\epsilon+\delta}{\epsilon}+\frac{\pi\delta(2\epsilon+\delta)}{\epsilon}\\
&\leq \frac{\epsilon+\delta}{\epsilon}\mi(K_n)+\frac{\pi\delta(2\epsilon+\delta)}{\epsilon},
\end{split}
\]
and taking the limit as $n\to\infty$ gives
\begin{equation*}
\frac{\leb(K^{\oplus\epsilon}\setminus K^{\oplus\delta})}{\epsilon}-\pi\epsilon\leq \frac{\epsilon+\delta}{\epsilon}\liminf_{n\to\infty} \mi(K_n)+\frac{\pi\delta(2\epsilon+\delta)}{\epsilon}
\end{equation*}
Letting $\delta\to 0$, since $\bigcup_{\delta>0}(K^{\oplus\epsilon}\setminus K^{\oplus\delta})=K^{\oplus\epsilon}\setminus K$, the continuity of the measure on increasing sequences of sets yields
\[
\frac{\leb(K^{\oplus\epsilon}\setminus K)}{\epsilon}-\pi\epsilon\leq \liminf_{n\to\infty} \mi(K_n).
\]
Taking the supremum in $\epsilon>0$, by Lemma~\ref{Mchar} we obtain the thesis.
\end{proof}

We use Theorem~\ref{Mlsc} to answer a question posed by Cerf in \cite{cerf}, where the following quantity has been studied in detail.

\begin{definition}[Hausdorff lower semicontinuous envelope of the perimeter] \label{envelope} 
\sloppy Let  $K\subset \mathbb R^2$ be a continuum. The \emph{Hausdorff lower semicontinuous envelope of the classical perimeter} $\mathcal S(K)$ of the set $K$ is defined as
\begin{equation*}
S(K):=\inf\big\{\liminf_{n\to\infty}\haus(\partial K_n):\, K_n \text{ Lipschitz continuum in $\R^2$},\, K_n\to^H K\big\}.
\end{equation*}
\end{definition}
In \cite{cerf} the Hausdorff lower semicontinuous envelope of the classical perimeter has been characterized as follows:
\begin{equation*}
\per(K)=\sup_{\mathcal U}\sum_{U\in\mathcal U}\sum_{O\in \mathcal C(K,U)}\haus(\partial O\setminus\partial U),
\end{equation*}
where $\mathcal C(K,U)$ is the collection of all residual domains of $K$ in $U$ and the supremum is taken over all families $\mathcal U$ of pairwise disjoint domains of $\R^2$. 
In the introduction of \cite{cerf} Cerf pointed out the interesting question of compare $\mathcal S$ with other classical quantities, like for instance the De Giorgi perimeter or the Minkowski content. 
In the following corollary we show that it coincides with the {outer} Minkowski content on continua of the plane.

\begin{corollary}\label{cor.cerf}
Let $K\subset \R^2$ be a continuum. Then the following equality holds:
\[
\per(K)=\mi(K).
\]
\end{corollary}
\begin{proof}
Let $\{K_n\}$ be a sequence of Lipschitz continua in $\R^2$ with $K_n\to^H K$. By Theorem~\ref{Mlsc} and \eqref{regular} it follows that
\[
\mi(K)\leq \liminf_{n\to\infty} \mi(K_n)=\liminf_{n\to\infty} \haus(\partial K_n).
\]
By taking the infimum among all sequences $\{K_n\}$ of Lipschitz continua in $\R^2$ with $K_n\to^H K$, recalling Definition~\ref{envelope}, we obtain the inequality $\mi(K)\leq \per(K)$.

For the reverse inequality we may assume $\mi(K)$ to be finite, otherwise the inequality is trivial. Then by the coarea formula  and Fatou's lemma (see (2.74) and Theorem 1.20 in \cite{amfupa}) we have 
\begin{equation}\label{proof111}
\begin{split}
\mi(K)&=\lim_{\epsilon\to0}\frac{1}{\epsilon}\int_0^\epsilon \haus (\{\mathrm{d}_K=t\}) \, dt\geq \int_0^1 \liminf_{\epsilon\to0}\haus (\{\mathrm{d}_K=t\epsilon\}) \, dt
\\ &\geq \liminf_{\epsilon\to0}\haus (\{\mathrm d_K=\epsilon\})\geq \liminf_{\epsilon\to0}\haus (\partial K^{\oplus\epsilon}),
\end{split}
\end{equation}
where the last inequality is a consequence of the inclusion  $\partial K^{\oplus \epsilon}\subseteq  \{\mathrm d_K= \epsilon\}$. For every $\epsilon>0$ the set $K^{\oplus\epsilon}\subset \R^2$ is a continuum (the connectedness follows from the one of $K$) and $K^{\oplus \epsilon}\to^H K$ as $\epsilon\to 0$. 
If the sets $K^{\oplus \epsilon}$ were Lipschitz then the inequality would follow by Definition~\ref{envelope}, but in general they are not. Nevertheless, we can conclude the proof by a diagonal argument, approximating, in the Hausdorff convergence, each set $K^{\oplus \epsilon}$ by means of smooth sets (we adapt the proof of \cite[Theorem 3.42]{amfupa} to our contest where a similar approximation has been provided, but with a different topology). Indeed, from \cite[Sect. 3.5]{amfupa} we have $\haus (\partial K^{\oplus\epsilon})\geq \mathcal P(K^{\oplus\epsilon})$ and similar to \cite[Theorem 3.42]{amfupa} we can build a sequence $\{K_n^\epsilon\}$ of Lipschitz continua  such that $\mathcal P(K_n^\epsilon)\to \mathcal P(K^{\oplus \epsilon})$ as $n\to\infty$ (here we can deal with closed instead open sets since the perimeter $\mathcal P$ of a set $E$ does not change by modifying $E$ with a set of  Lebesgue measure zero and from \eqref{proof111} $\leb(\partial K^{\oplus\epsilon})=0$ while from the regularity of $K_n^\epsilon$ it holds $\leb(\partial K_n^\epsilon)=0$). Therefore, if we prove that $K_n^\epsilon\to^H K^{\oplus \epsilon}$ as $n\to\infty$, by a standard diagonal argument (by the Blaschke selection theorem \cite[Theorem 4.4.15]{ambtil} the Hausdorff convergence over closed sets is metrizable), 
there exists a subsequence $\{K_{n_j}^{\epsilon_j}\}$ of Lipschitz continua with $K_{n_j}^{\epsilon_j}\to^H K$ as $j\to \infty$ such that 
\eqref{proof111} and Definition~\ref{envelope} yield
\[
\mi(K)\geq \liminf_{j\to \infty}\haus(\partial K_{n_j}^{\epsilon_j})\geq \per(K),
\]
where we also used  \eqref{regular} to say that $\mathcal P(K_{n^j}^{\epsilon_j})=\haus(\partial K_{n^j}^{\epsilon_j})$. 

To conclude it remains to prove the convergence of smooth sets to $K^{\oplus\epsilon}$. Let $\epsilon>0$ be fixed. Similar to \cite[Theorem 3.42]{amfupa}, we can choose the closed counterpart $K_n^\epsilon:=\{\chi_{K^{\oplus \epsilon}}*\rho_n\geq t\}$ where $\rho_n$ is a mollifier with support in the ball $B(0,1/n)$ and $t$ is a suitable fixed real number with $t\in(0,1)$.  Relying on this definition we can prove that $K_n^\epsilon\to^H{K^{\oplus \epsilon}}$ as $n\to\infty$, namely that for every $\delta>0$ there exists $n_\delta$ such that for every $n>n_\delta$ the inclusions $K_n^\epsilon\subset (K^{\oplus\epsilon})^{\oplus \delta}$ and $K^{\oplus\epsilon}\subset (K_n^\epsilon)^{\oplus \delta}$ hold.
Fix $\delta>0$ and let $n_\delta=\lceil 1/\delta\rceil$. For the former inclusion, since $\chi_{K^{\oplus \epsilon}}*\rho_n$ is a non-negative function, for every $n>n_\delta$ we have
\[
K_n^\epsilon\subset \mathrm{supp}(\chi_{K^{\oplus \epsilon}}*\rho_n)\subset \mathrm{supp}(\chi_{K^{\oplus \epsilon}})+\mathrm{supp}(\rho_n)=(K^{\oplus \epsilon})^{\oplus 1/n}\subset (K^{\oplus \epsilon})^{\oplus \delta},
\]
where $\mathrm{supp}$ denotes the support of a function and $+$ the Minkowski sum. For the latter inclusion let $s=1-t$ and notice that $1-\chi_{K^{\oplus \epsilon}}*\rho_n=(1-\chi_{K^{\oplus \epsilon}})*\rho_n$ is again a non-negative function. Then, similarly to the previous inclusions
\[
\mathbb R^2\setminus K_n^\epsilon\subset 
\mathrm{supp}((1-\chi_{K^{\oplus \epsilon}})*\rho_n)\subset \mathrm{supp}(1-\chi_{K^{\oplus \epsilon}})+\mathrm{supp}(\rho_n)=(\mathbb \R^2\setminus K^{\oplus \epsilon})^{\oplus 1/n}.
\]
Since $(\R^2\setminus K^{\oplus \epsilon})^{\oplus 1/n}\subset\R^2\setminus  K^{\oplus (\epsilon-1/n)}$ by passing to the complementary sets and enlarging of $\delta$ we obtain $(K_n^\epsilon)^{\oplus \delta}\supset (K^{\oplus (\epsilon-1/n)})^{\oplus \delta}$. If $n>n_\delta$ it is easy to see that $(K^{\oplus (\epsilon-1/n)})^{\oplus \delta}\supset K^{\oplus \epsilon}$ and the latter inclusion above is also proved. This concludes the proof of the corollary.
\end{proof}

{This corollary may serve as a translator of }similar results independently established for the Hausdorff lower semicontinuous envelope of the perimeter $\per$ in \cite{cerf}  and for the outer Minkowski content $\mi$ in \cite{vil}. In particular, from Corollary~\ref{cor.cerf} one could deduce that, if $K\subset \mathbb R^2$ is a continuum then
\[
\per(K)=\mathcal P(K)+2\haus({\partial K\cap K^0}),
\]
where $K^0$ is the set of points where $K$ has null density (see \cite{cerf, vil}).

In the following we use Corollary~\ref{cor.cerf} to show that the convex hull diminishes the outer Minkowski content of a continuum in the plane and moreover that the outer Minkowski content controls the Hausdorff measure.

\begin{corollary}\label{cor.hull}
Let $K\subset \mathbb R^2$ be a continuum and $\hull(K)$ be the convex hull of $K$. Then 
\[
\mi(\hull(K))\leq \mi(K).
\]
\end{corollary}
\begin{proof}
For a Lipschitz continuum of the plane, it is well-known that the convex hull diminishes the perimeter (see \cite{ferfus} and recall \eqref{regular}). For a general continuum, we proceed by approximation using the characterization of the outer Minkowski content provided in Corollary~\ref{cor.cerf}. Indeed, for every $\eta>0$ we may consider a sequence of Lipschitz continua $\{K_n\}$ such that $K_n\to^H K$ and $\liminf_n \haus (\partial K_n) \leq \mi(K)+\eta$. Since the convex hull is stable with respect to the Hausdorff convergence, that is $\hull(K_n)\to^H \hull(K)$  (see \cite[Exercise 2.5]{henpie}), by Corollary~\ref{cor.cerf} and Definition~\ref{envelope} we obtain
\[
\mi(\hull(K))\leq \liminf_{n\to\infty} \haus(  \partial \hull(K_n))\leq  \liminf_{n\to\infty} \haus(\partial K_n)\leq \mi(K)+\eta.
\]
We get the thesis by the arbitrariness of $\eta$ .  
\end{proof}

\begin{corollary}\label{cor.residual}
Let $K\subset \R^2$ be a continuum with a finite number of residual domains in $\mathbb R^2$. Then
\begin{equation*}%\label{previous.ineq}
\haus(\partial K)\leq \mi(K).
\end{equation*}
\end{corollary}
\begin{proof}
Let $\eta>0$ and denote by $k\in\mathbb N$ the number of residual domains of $K$  in $\mathbb R^2$. From Corollary~\ref{cor.cerf} there exists a sequence of Lipschitz continua $\{K_n\}$ with $K_n\to^H K$ and
\begin{equation}\label{iooioio}
\lim_{n\to\infty} \haus(\partial K_n)\leq \mi(K)+\eta.
\end{equation}
(we choose a subsequence for which the $\liminf$ is a limit).
From the assumption on $K$ each set $K_n$ of the approximating sequence can be chosen to have at most $k$ residual domains in $\mathbb R^2$. This implies that $\partial K_n$ has at most $k$ connected components as well (otherwise its perimeter would be larger). Now, the sequence of closed sets $\{\partial K_n\}$  converges, up to subsequences (not relabelled), to a closed set $J$ with at most $k$ connected components and so that $\partial K\subseteq J\subseteq K$ (the first inclusion follows from Kuratowski convergence, the second one since the inclusion is continuous with respect to the Hausdorff convergence). Therefore, by the monotonicity of the Hausdorff measure and by applying the Go\l ab theorem \cite[Theorem 4.4.17]{ambtil} to each connected component of $\partial K_n$, we obtain that
\[
\haus(\partial K)\leq\haus(J)\leq \liminf_{n\to\infty} \haus(\partial K_n).
\]
Combining this inequality with \eqref{iooioio}, by the arbitrariness of $\eta$, we get the thesis.
\end{proof}

\section{Optimal obstacles: toward existence, regularity and geometry.}\label{sec.3}

We start by proving the existence of a solution to problem \eqref{problem}. Then, we analyze some qualitative properties satisfied by such a solution.

\begin{theorem}
Let $L\in(0,\mi(\overline \Om))$. Then there exists a maximizer $\K$ of \eqref{problem}.
\end{theorem}

\begin{proof}
The existence follows from the direct methods of the Calculus of Variations. Let $\{K_n\}$ be a maximizing sequence of problem \eqref{problem}, so that
\begin{equation}\label{sup}
\la_1(\Om\setminus K_n)\to\sup\{\la_1(\Om\setminus K):\: \text{$K\subseteq \overline{\Om}$, $K$ continuum, $\mi(K)\leq L$}\},
\end{equation}
as $n\to\infty$. By the Blaschke selection theorem there exists a compact set $\K\subset \overline{\Om}$ and a subsequence, not relabelled, such that $K_n\to^H \K$. 
Moreover, by Theorem~\ref{Mlsc}, the set $\K$ is a continuum with $\mi(\K)\leq L$ and thus it is an admissible competitor in \eqref{problem}. Then, by the Sverak continuity result (see \cite{sve}) $\la_1(\Om\setminus K_n)\to \la_1(\Om\setminus \K)$ as $n\to \infty$. This with \eqref{sup} implies that $\K$ solves problem \eqref{problem}. 
\end{proof}

Before giving the first properties of a solution to \eqref{problem}, let us introduce the notion of \emph{local convexity}.

\begin{definition}[Local convexity]%\label{locconvex}
A continuum $K\subset \overline{\Om}$ is said to be \emph{locally convex inside $\Om$} if for every $x\in \partial K\cap \Om$ there exists $r_x> 0$ such that $K\cap \overline{B(x,r_x)}$ is convex.
\end{definition}

Notice that if this definition is satisfied for some radius $r_x>0$ then it is also satisfied for all smaller radius $0<r<r_x$. Moreover, for points in $(K\setminus \partial K)\cap \Om$ there always exists such a radius $r_x$. An interesting result relating local to global convexity goes back to 1928 (see \cite{tie,nak}).

\begin{theorem}[Tietze-Nakajima]\label{tienak}
Let $K\subset \mathbb R^2$ be a continuum that is locally convex inside $\mathbb R^2$. Then $K$ is convex.
\end{theorem}

We also need the following lemma.

\begin{lemma}\label{lemma.minor2}
Let $L>0$ and let $K\subset \overline{\Om}$ be a Lipschitz continuum with $\haus(\partial K)\leq L$. Let $x\in \Om$ and let $r>0$ such that $\overline{B(x,r)}\cap \overline{\Om}$ is convex. Then there exists a continuum $\widehat K$ with the following properties:
\begin{itemize}
\item[-] $\widehat K\subset \overline \Om$, $\widehat K\supset K$ and $\haus(\partial \widehat K)\leq \haus(\partial K)$;
\item[-] $\widehat K \cap \overline{B(x,r/2)}$ is the union of $N$ convex continua with $N\leq \lfloor L/r \rfloor$.
\end{itemize}
\end{lemma}
\begin{proof}
The first condition is trivially satisfied if $\widehat K=K$, but the second one is not in general. Therefore, we construct a set $\widehat K$ satisfying both conditions by taking suitable convex hulls of subsets of $K$.
Let $\{K^i\}$ be the family of those connected components of $K \cap \overline{B(x,r)}$ with non empty intersection with the disk $\overline{B(x,r/2)}$.
We first notice that, for any index $i$, the set $H:=K\cup H^i$ where $H^i:=\hull(K^i)$ is a continuum such that
\begin{equation}\label{areaperimeter}
H\subset \overline \Om, \quad H\supset K, \quad \haus(\partial H)\leq \haus(\partial K).
\end{equation}
The inclusions follow from the fact that $K^i\subset H^i\subset \overline{B(x,r)}\cap \overline{\Om}$. For the inequality on the Hausdorff measure, by writing $\partial K^i=(\partial K^i\setminus \partial H^i)\cup (\partial K^i\cap \partial H^i)$ and $\partial H^i=(\partial H^i\setminus \partial K^i )\cup (\partial H^i\cap \partial K^i)$,
since $\haus(\partial H^i)\leq\haus (\partial K^i)$ (see \cite{ferfus}), we deduce that
\begin{equation}\label{estimateone}
\haus(\partial H^i\setminus \partial K^i)\leq\haus(\partial K^i\setminus \partial H^i).
\end{equation}
Similarly, we write $\partial K=(\partial K\setminus \partial H )\cup (\partial K\cap \partial H)$ {and}  $\partial H=(\partial H\setminus \partial K )\cup (\partial H \cap \partial K)$.
Since $K^i\subset H^i$ and $K\subset H$ we obtain
\[
\partial K^i\setminus \partial  H^i=\partial K^i\cap \mathrm{int}({H^i})\subset \partial K\cap \mathrm{int}({H^i})\subset \partial K\cap \mathrm{int}({H})=
\partial K\setminus \partial H,
\]
and the monotonicity of $\haus$ with \eqref{estimateone} yield
\begin{equation}\label{estimatetwo}
\haus(\partial H^i\setminus \partial K^i)\leq \haus(\partial K^i\setminus \partial H^i)\leq\haus(\partial K\setminus \partial H).
\end{equation}
The inclusion $\partial H\setminus \partial K\subset \partial H^i\setminus\partial K^i$ holds. Indeed, let $x\in\partial H\setminus \partial K$, that is $x\in K\cup H^i$ but $x\notin \mathrm{int}(K\cup H^i)$ and $x\notin \partial K$. 
Then $x\notin \mathrm{int}(K)$ and $x\notin\mathrm{int}(H^i)$. These facts imply that $x\in \partial H^i$ and moreover that $x\notin \partial K^i$.  Therefore, using again the monotonicity of $\haus$ yields $\haus(\partial H\setminus \partial K)\leq\haus(\partial H^i\setminus\partial K^i)$ that combined with the inequality \eqref{estimatetwo} gives \eqref{areaperimeter}. 

Now, by using the monotonicity and the additivity properties of $\haus$ with the connectedness of each component we deduce that the cardinality $N$ of the family $\{K^i\}$ is at most $\lfloor{L}/{r}\rfloor$. Therefore, setting $\widehat K:=\bigcup_{i=1}^N H^i\cup K$ provides a continuum satisfying the properties listed in the statement of theorem. Indeed, the first point of the list follows by induction from \eqref{areaperimeter}. %(each times we add a set $H^i$ either the number of \emph{connected} components or the number of \emph{non convex} components decreases).
For the second one, notice that
\[
\widehat K\cap \overline{B(x,r/2)}=\Big(\bigcup_{i=1}^N H^i\cup K^i\Big)\cap \overline{B(x,r/2)}=\bigcup_{i=1}^N H^i\cap \overline{B(x,r/2)},
\]
that is a finite union of convex continua.
\end{proof}

\begin{theorem}\label{teo.opt}
Let $L\in(0,\mi(\overline \Om))$ and let $\K$ be a maximizer of \eqref{problem}. Then the following properties hold.
\begin{description}
\item[(i)] $\K$ is locally convex inside $\Omega$. Moreover, if $\Om$ is convex then $\K$ is convex as well.
\item[(ii)] The perimeter constraint is saturated, namely $\mi(\K)=L$.
\item[(iii)] If $\overline{\Om}$ has $k$ residual domains in $\mathbb R^2$ with $k\geq 1$  
then $\K$ has at most $k$ residual domains in $\mathbb R^2$.
\end{description}
\end{theorem}
\begin{proof}
In the proofs of the three items we proceed by contradiction.

\smallskip
For (i) assume $\K$ to be not locally convex inside $\Om$. This guarantees the existence of a point $x\in \partial \K\cap \Om$ for which the closed set $\K\cap \overline{B(x,r)}$ is not convex for all $r>0$. Then fix some $r$ so small so that $\overline{B(x,r)}\cap \overline{\Om}$ is convex. The strategy of the proof consists in considering as competitor the continuum $\K\cup \hull(K^1)$ where $K^1$ is the connected component of $\K\cap \overline{B(x,r/2)}$ containing $x$. If $\K\cup \hull(K^1)$ were admissible then, by monotonicity of the first eigenvalue, it would contradict the optimality of $\K$. However, due to the subadditivity of the outer Minkowski content, it is not immediate to reach a contradiction. To overcome this difficulty, by approximation and the technical lemma above, we deduce more information on the set $\K\cap \overline{B(x,r/2)}$ that garantee to have additivity of the outer Minkowski content.

Let $\eta>0$. From Corollary~\ref{cor.cerf} and the estimate $\mathcal{SM}(\K)\leq L$, there exists a sequence of Lipschitz continua $\{K_n\}$ such that $K_n\to^H \K$ and 
\begin{equation*}%\label{boundedbyL}
\haus (\partial K_n)\leq L+\eta,
\end{equation*}
for every $n$ large enough. Without loss of generality, by optimality of $\K$ and Lemma~\ref{lemma.minor2} we can assume $K_n \cap \overline{B(x,r/2)}$ to be the union of at most $\lfloor (L+\eta)/r \rfloor$ convex continua (notice that this bound  is uniform with respect to $n$). Then, since convexity is preserved by Hausdorff convergence, it turns out that the set $\K\cap \overline{B(x,r/2)}$ is the union of a \emph{finite} number of convex continua and, possibly, a non empty subset of $\partial B(x,r/2)$ (this represent the set of those limit points $x$ for which there exists a sequence $\{x_n\}$ such that $x_n\in K_n\setminus \overline{B(x,r/2)}$ and $x=\lim_{n\to\infty} x_n$). 
Now, let $K^1$ be the connected component of $\K\cap \overline{B(x,r/2)}$ containing $x$ and let $K^2:=\K\cap \hull(K^1)\setminus K^1$, that is the intersection of the remaining connected components with $\hull(K^1)$.  
Since $\hull(K^1)\setminus K^1\subset {B(x,r/2)}$ we deduce that $K^2$ has a \emph{finite} number of connected components, thus is compact and the quantity
\begin{equation}\label{wellseparated}
d:=\min_{x\in K^1,\, y\in K^2} |x-y|>0.
\end{equation}
By choosing the radius $\epsilon$ in the definition of $\mi^*(K^1\cup K^2)$ (see Definition~\ref{minkowski}) smaller than the number $d$ as defined in \eqref{wellseparated} and by using Lemma~\ref{Mchar} we obtain the \emph{additivity} of $\mi^*$
\begin{equation}\label{additivity}
\mi^*(\K\cap \hull(K^1))=\mi^*(K^1\cup K^2)= \mi(K^1)+\mi^*(K^2).
\end{equation}
Now, we notice that $K^1$ can not be convex (otherwise for $r<d$ the set $\K\cap \overline{B(x,r)}=K^1\cap  \overline{B(x,r)}$ would be convex, a contradiction with the assumption at the beginning of the proof).
Therefore, the set $\widehat{K}:=\K\cup \hull(K^1)$ strictly contains $\K$ and by monotonicity of the first eigenvalue $\lambda_1(\Omega\setminus \widehat K)>\lambda_1(\Omega\setminus \K)$. By convexity of $\overline{B(x,r)}\cap \overline\Om$ the inclusion $\hull(K^1)\subset \overline{B(x,r)}\cap\overline\Omega$ holds so that $
\widehat K\subset\overline{\Om}$, and this implies that $\mi(\widehat K)>\mi(\K)$ 
(otherwise $\widehat K$ would be a better competitor contradicting the optimality of $\K$).
Plugging this inequality into the \emph{strong subadditivity} of $\mi$ (that can be proved similarly to \cite[p. 739]{amcovi}))
\begin{equation}\label{strongsub}
\mi(\widehat K)+ \mi^*(\K\cap \hull(K^1))\leq \mi(\K)+\mi(\hull(K^1)),
\end{equation}
yields 
$\mi^*(\K\cap \hull(K^1))< \mi(\hull(K^1))$.
This inequality with \eqref{additivity} and Corollary~\ref{cor.hull} applied to $K^1$ gives $\mi^*(K^2)<0$, a contradiction. The set $\K$ must be locally convex inside $\Om$.

The second statement in the case $\Om$ convex follows directly from Corollary~\ref{cor.hull}: $\K$ must be convex otherwise by convexity of $\Om$ the continuum $\hull(K)$ would contradict the optimality of $\K$.

\smallskip
For (ii) let us now assume $\mi(\K)<L$. Consider a segment $\gamma$  so that $\K\cup \gamma$ is connected and $0<\mi(\gamma\setminus \K)\leq L-\mi(\K)$ (this is always possible thanks to \eqref{curve}).Then, by the subadditivity of $\mi$ with respect to union of sets (similar to \eqref{strongsub}) we deduce that $\mi(\K\cup \gamma)\leq L$ while $\la_1(\Omega\setminus(\K\cup\gamma))>\la_1(\Omega\setminus \K)$. This contradicts the optimality of $\K$ and so every maximizer of \eqref{problem} saturates the perimeter constraint.

\smallskip
At last for (iii) let us assume that $\K$ has more than $k$ residual domains in $\mathbb R^2$. Since at most $k$ different residual domains of $\K$ can contain those of $\overline \Om$, all the other ones must be contained in $\Omega$. Let $\omega\subset \Omega$ be such a residual domain of $\K$ and define the continuum $\widehat K=\K\cup \omega$. Then $\widehat K\subset \overline{\Omega}$, $\widehat K\supset \K$ and since $({\widehat K}^{\oplus\epsilon} \setminus \widehat K)\subset (\K^{\oplus\epsilon} \setminus \K)$, we also have $\mi(\widehat K)\leq \mi(\K)$. This provides a better competitor in \eqref{problem} contradicting the optimality of $\K$. Therefore, $\K$ has at most $k$ residual domains. 
\end{proof}

\begin{remark}
If $\overline{\Omega}$ is Lipschitz then the local convexity of $\K$ can be proved \emph{up to the boundary $\partial \Om$}.
Precisely, for every vertex $x\in \partial \K$, height $r>0$, direction $\xi\in \mathbb S^1$, and opening angle $\theta\in (0,2\pi]$ such that the intersection of the cone $\overline{C(x, r , \xi, \theta )}$ with $\overline{\Omega}$ is convex (the existence of a 4-tuple satisfying this condition is guaranteed by the regularity of $\Om$) then the set $\K\cap \overline{C(x,r,\xi,\theta)}$ is convex as well. This can be proved simply by replacing in the proofs above disks with cones.
\end{remark}
\begin{remark}
We will see in the situation of the ring, where $\overline{\Omega}$ has exactly 2 residual domains in $\mathbb R^2$, that the two possibilities $\K$ has 1 or 2 residual domains in $\mathbb R^2$ actually happen, depending on the perimeter constraint (see Theorem~\ref{casering} below).
\end{remark}

Now, we focus on the regularity and the geometry of the \emph{free boundary} of $\K$ (the part of the boundary which is inside $\Om$).

\begin{theorem}\label{teo.reg}
The free boundary of $\K$ is of class $C^\infty$, that is for every $x\in \partial \K\cap \Om$ there exists $r_x>0$ such that $\partial \K\cap \overline{B(x,r_x)}$ is the graph of a concave $C^\infty$ function.
Moreover, for every $x\in \partial \K\cap \Om$ the following properties hold.
\begin{description}	
\item[(i)] If $x\in \partial\omega$, with $\omega$ a residual domain of $\K$ in $\Omega$ and $\la_1(\omega)>\la_1(\Omega\setminus\K)$, then  $\partial \K\cap \overline{B(x,r_x)}$ is a segment.
\item[(ii)] If $\partial \K\cap \overline{B(x,r_x)}$ is not a segment then the \emph{optimality condition} holds
\begin{equation}\label{optimality}
|\nabla u_1(x)|^2=\mu  \mathcal{C}(x),
\end{equation}
where $u_1$ is the first eigenfuction corresponding to $\la_1(\Om\setminus \K)$, $\mathcal{C}(x)$ is the curvature of the free boundary of $\K$ at the point $x$
and $\mu>0$ is a Lagrange multiplier which may depend on the connected component of $\partial\K \cap\Om$ but is the same for all the points on this
connected component.
\item[(iii)] If $\partial \K\cap \overline{B(x,r_x)}$ is an arc of circle then $\Omega\setminus \K$ is a ring. In particular, if $\K$ is a disk then $\Omega$ is a disk concentric to $\K$.
\end{description}	
\end{theorem}

\begin{proof}
Fix $x\in \partial \K\cap \Om$ and let $\gamma:=\partial \K\cap \overline{B(x,r_x)}$ be the graph of a concave function for some fixed radius $r_x>0$. This is always possible thanks to item (i) of Theorem \ref{teo.opt}. The regularity of $\gamma$ is quite classical: this follows by using a Schauder's regularity result with a bootstrap argument (see, e.g., \cite{chalar} or also the proof of Theorem 2.2 in \cite{bubuhe}). 
In our problem (by contrast with what happens in \cite{bubuhe})
the domain is not convex but is the complement of a locally convex set inside $\Omega$.
Nonetheless, classical regularity results in the plane imply that $|\nabla u_1|^2\in L^p(\gamma)$ for some $p>1$ (see for instance \cite{jerken}).  This is enough to
start the bootstrap argument and to follow the same line as in \cite{bubuhe} to get the regularity of the free boundary.
\smallskip

For (i) if it is not the case then $\gamma$ is strictly convex somewhere and it is possible to decrease the perimeter of this connected component without changing the eigenvalue $\lambda_1(\Omega\setminus \K)$ contradicting item (ii) in Theorem~\ref{teo.opt}. 

\smallskip
To prove (ii), 
by the previous point (i), observe that the point $x\in \partial\omega$ for some residual domain $\omega$ of $\K$ in $\Omega$ with $\la_1(\omega)=\la_1(\Omega\setminus\K)$. 
Now, assume that $\gamma$ is modified by a regular vector field $x\in \R^2\mapsto I(x)+tV(x)$ where $t>0$, $I$ is the identity map from $\R^2$ to $\R^2$ while $V\in C^2(\R^2;\R^2)$ has compact support inside $B(x,r_x)$. Then by \cite[Proposition 4.13]{vil} we can use the shape derivative of the classical perimeter \cite[Corollary 5.4.16]{henpie} to obtain that
\[
\mi\left((I+tV)(\K)\right)=\mi(\K)+t\int_{\gamma}\mathcal C\,V. n\, d\haus +o(t),\quad  \text{as $t\to 0$},
\]
where $n$ is the normal to $\gamma$ pointing toward $\Omega\setminus \K$ (which exists everywhere by regularity of $\gamma$).
Since by assumption the curvature $\mathcal C$ is \emph{positive} over a subset of $\gamma$  we can always consider vector fields $V$ decreasing the total perimeter of $\K$, namely so that
\begin{equation}\label{variationp}
\int_{\gamma} \mathcal C\,V.n\, d\haus \leq 0.
\end{equation}
On the other hand, the shape derivative of the first eigenvalue \cite[Theorem 5.7.1]{henpie} implies the following expansion
\begin{equation}\label{expansione}
\lambda_1\left((I+tV)(\omega)\right)=\lambda_1(\omega)+t \int_{\gamma}|\nabla u_1|^2 \,V. n \,d \haus +o(t),\quad  \text{as $t\to 0$},
\end{equation}
here the $+$ in the linear term is due to the definition of the normal $n$ as exterior to $\gamma$ (i.e., it points toward $\omega$). 

Let us now consider all the other residual domains $\{\omega_i\}$ of $\K$ in ${\Om}$, if they exist, such that $\la_1(\omega_i)=\la_1(\Om\setminus\K)$. Inside each of these components $\omega_i$ we can add a small segment to $\K$ by preserving the total perimeter (similarly to the proof of (ii) in Theorem~\ref{teo.opt}). Since the first eigenvalue of each of these components increases, by optimality of $\K$ we must have $\lambda_1(\Omega\setminus (I+tV)(\K))=\lambda_1\left((I+tV)(\omega)\right)$ so that \eqref{expansione} yields
\begin{equation}\label{variatione}
\int_{\gamma} |\nabla u_1|^2 \,V.n \,d \haus \leq 0.
\end{equation}
In other words, for any vector field $V$ such that \eqref{variationp} holds it follows the inequality \eqref{variatione}.
This shows that the two linear forms $V\mapsto \int_{\gamma} \mathcal C\,V. n\, d\haus$
and $V \mapsto \int_{\gamma} |\nabla u_1|^2 \,V. n \,d \haus$ are proportional and the optimality condition \eqref{optimality} holds.

\smallskip
The assertion (iii) follows from the optimality condition \eqref{optimality} in the same way as in \cite{bubuhe} (see also \cite{hen-oud}) where a
similar statement has been proved for the second eigenvalue of the Laplacian. More precisely, if the boundary contains an arc of circle $\gamma$ centered at the origin, introducing the function $w:=x {\partial u_1}/{\partial y} - y {\partial u_1}/{\partial x}$ (i.e., the derivative of $u_1$ with respect to the angular coordinate) we can prove using \eqref{optimality} that $w$ satisfies $-\Delta w=\lambda w$ in $\Om\setminus\K$ with $w=\frac{\partial w}{\partial n}=0$ on $\gamma$. By H\"olmgren uniqueness theorem (see \cite[Proposition 4.3]{tay}) and analyticity, this implies that $w\equiv 0$ in
$\Om\setminus\K$. But this means that $u_1$ is radially symmetric in $\Om\setminus\K$ and then $\Om\setminus\K$ has to be a ring.
\end{proof}

\begin{remark}
The optimality condition \eqref{optimality} shows that the curvature of the free boundary is positive everywhere and expresses, in a quantitative way, the fact that a maximizer of \eqref{problem} has to be locally convex inside $\Omega$, see (i) in Theorem~\ref{teo.opt} (cf. with the optimality conditions obtained in \cite{bubuhe,mazzuc}).

In the next section, we will study  more in detail item (iii) for specific domains $\Omega$. These include the cases where $\Omega$
is a disk, a ring or more generally a disk with convex holes. In these situations, we will identify the maximizer for certain values of $L$ (actually for all values of $L$ when $\Omega$ is itself a disk).
\end{remark}

The intuition may lead to think that a maximizer of \eqref{problem} must always stay inside $\Omega$, see for example the situation
described in \cite{hakrku} where the maximizing position is at the center of the domain while only in the
minimizing positions the obstacle touches the boundary. This is probably true when $\Omega$
is convex, but we were not able to prove it. On the other hand, when $\Omega$ is not convex, we prove that
it is never the case when $L$ is large enough and that $\K$ must \emph{touch the boundary of} $\Omega$. To show this we rely on an object that measure the largest perimeter one can reach by means of convex subsets of $\Om$.

\begin{proposition}
The following quantity 
\begin{equation}\label{clen}
\clen(\Om):=\max\{\mi(K):\; \text{$K\subseteq\overline{\Om}$, $K$ closed and convex}\}
\end{equation}
is well defined and 
\begin{equation*}
\clen(\Om)\leq \mi(\overline{\Om}),
\end{equation*}
where the equality holds if and only if $\Om$ is convex. 
\end{proposition}
\begin{proof}
Let $\{K_n\}$ be a maximizing sequence of problem \eqref{clen}, so that
\begin{equation*}
\mi(K_n)\to\sup\{\mi(K):\: \text{$K\subseteq \overline{\Om}$, $K$ closed and convex}\},
\end{equation*}
as $n\to\infty$.
From the Blaschke selection theorem, we infer the existence of a compact set $K^*\subset \overline{\Om}$ and a subsequence, not relabelled, such that $K_n\to^H K^*$. Moreover, since the Hausdorff convergence preserves convexity (see \cite{bucbut}) the limit $K^*$ is also convex in $\overline{\Om}$. Therefore, using \eqref{regular} with the continuity of the classical perimeter with respect to the Hausdorff convergence of convex sets we obtain the existence of a solution  to \eqref{clen} and $\clen(\Om)=\mi(K^*)$.

Now, $K^*$ is a convex set included in $\hull(\overline{\Om})$ thus $\haus(\partial K^*)\leq \haus(\partial \hull(\overline \Om))$. Therefore, by \eqref{regular} again, with Corollary~\ref{cor.hull} we obtain $\mi(K^*)\leq \mi(\overline{\Om})$,
{and the equality holds whenever $\Om$ is convex. }
\end{proof}

\begin{theorem}
If $L\in\big(\clen(\Om),\mi(\overline \Om)\big)$ with $\clen(\Om)$ defined by \eqref{clen}, then every maximizer $\K$ of \eqref{problem} touches the boundary $\partial \Om$, i.e., $\K\cap\partial \Om\neq \emptyset$. 
\end{theorem}
\begin{proof}
Let us assume, for a contradiction, that $\K\cap\partial \Om= \emptyset$ so that $\K\subset \Om$.  Then (i) of Theorem~\ref{teo.opt} implies that the set $\K$ is locally convex inside $\R^2$ and by Theorem~\ref{tienak} $\K$ is convex. This with item (ii) in Theorem~\ref{teo.opt} would contradict the assumption $L>\clen(\Om)$. 
\end{proof}

\section{Optimal obstacles in specific domains}\label{sec.4}

We study problem \eqref{problem} for specific domains $\Om$: circular, annular, and perforated domains. For these domains we prove symmetry and, in some cases non symmetry results, identifying the unique solution.

\subsection{Circular domains}

We identify the maximizer of \eqref{problem} in the case the domain $\Om$ is a disk. Our argument relies on the following result obtained in the sixties by Hersch, Payne, and Weinberger (see  \cite{her1,her2, paywei} and also \cite[Section~3.5]{hen} for a concise explanation of these papers).

\begin{theorem}[Hersch-Payne-Weinberger]\label{hpw}
Let $D$ be a doubly connected domain of the plane (i.e., a domain bounded between two disjoint and rectifiable Jordan curves)  with outer boundary $\Gamma_0$ and inner boundary $\Gamma_1$ of length respectively $L_0$ and $L_1$. If
\begin{equation}\label{cond}
L_0^2-L_1^2=4\pi \leb(D),
\end{equation}
then the the first eigenvalue $\la_1(D)$ is uniquely maximized whenever $D$ is the ring with outer boundary of length $L_0$ and inner boundary of length $L_1$.
\end{theorem}

In the theorem above the competitors have free both the inner and outer boundaries, but perimeter and area are strongly
constrained. In our problem \eqref{problem} the exterior boundary is fixed and only the interior boundary is free to move
but there is no constraint on the area of $K$. 
We develop a purely geometrical argument in order to fit into the hypothesis of the Hersch, Payne, and Weinberger result: as a consequence we identify the explicit solution to \eqref{problem} when $\Om$ is a disk.

\begin{theorem}\label{casedisk}
Let $\Om=B(r_0)$  be an open disk of radius $r_0>0$, and $L\in(0,2\pi r_0)$. Then problem \eqref{problem}
has a unique solution: the maximizer $\K$ is given by the closed disk $\overline{B(r)}$ concentric to $B(r_0)$ of radius $r=L/(2\pi)$.
\end{theorem}

\begin{proof}
We look for a solution to \eqref{problem} only among {convex} sets with $\mi(K)=L$, since non-convex sets and sets with outer Minkowski content less than $L$ are ruled out by items (i) and (ii) in Theorem~\ref{teo.opt}. Therefore, it suffices to prove that for every closed and convex set $K$ contained in $\overline{B(r_0)}$ with $\mi(K)=L$, different from the disk $\overline{B(r)}$ concentric to $B(r_0)$ and with perimeter $L$, it holds 
\begin{equation}\label{ineq}
\la_1(B(r_0)\setminus K) <\la_1(B(r_0) \setminus \overline{B(r)}).
\end{equation}
We prove this inequality by exploring four cases, according to the shape and the location of the convex set $K$. 

\smallskip
\emph{Case 1: $K$ is a disk not concentric to $B(r_0)$.} In this case the inequality \eqref{ineq} is an easy consequence of (iii) in Theorem~\ref{teo.reg} 

\smallskip
\emph{Case 2: $K$ is neither a disk nor a segment and it is contained in $B(r_0)$}. 
Now we use Theorem~\ref{hpw}. 
Clearly for the disk $B(r_0)$ the condition $(2\pi r_0)^2=4\pi \leb(B(r_0))$ holds. Moreover, for the set $K$, by recalling  \eqref{regular}, the \emph{isoperimetric inequality} implies that  
\[
L^2=\haus(\partial K)^2> 4\pi \leb(K),
\]
which combined with the previous equality for $B(r_0)$ yields  
\begin{equation}\label{inequality}
(2\pi r_0)^2-L^2<4\pi \leb (B(r_0)\setminus K).
\end{equation}
This means that we can not apply Theorem~\ref{hpw} to the doubly connected domain $D=B(r_0)\setminus K$ with $L_0=2\pi r_0$ and $L_1=L$, since the equality condition \eqref{cond} is not satisfied. However it is possible to modify the set $B(r_0)\setminus K$, increasing its outer perimeter $L_0$ and decreasing its area $\leb (B(r_0)\setminus K)$ until the equality in \eqref{inequality} is reached. More precisely, starting from the disk $B(r_0)$, we consider a smooth domain $\widehat{B}\subset\R^2$ such that 
\begin{itemize}
\item[(i)] the perimeter increases: $\haus(\partial B(r_0))< \haus(\partial \widehat{B})$;
\item[(ii)] the set is smaller: $B(r_0)\supset \widehat{B}$; 
\item[(iii)] the equality condition holds: $\haus(\partial \widehat{B})^2-L^2=4\pi \leb (\widehat{B}\setminus K)$.
\end{itemize}
An explicit construction of the set $\widehat{B}$ can be obtained, for instance by perturbing the whole boundary of the disk $B(r_0)$ with an inward pointing vector field that continuously increases the perimeter and decreases the set (in the sense of set inclusion). Moreover, it is not difficult to see that the perimeter in point (i) can be made arbitrarily large until point (iii) is satisfied (and of course point (ii) contributes in this direction).

Therefore, from point (ii) and the strict monotonicity of the first eigenvalue with respect to set inclusion, we obtain the estimate
\begin{equation}\label{dim31}
\la_1(B(r_0)\setminus K)< \la_1(\widehat{B}\setminus K).
\end{equation}
Thanks to (iii) we can now apply Theorem~\ref{hpw} to the set $D=\widehat{B}\setminus K$ with $L_0=\haus(\partial \widehat{B})$ and $L_1=L$ so that
\begin{equation}\label{dim32}
\la_1(\widehat{B}\setminus K)< \la_1(B(\widehat r_0)\setminus \overline{B(r)}),
\end{equation}
where $B(\widehat r_0)$ is the open disk concentric to $B(r_0)$ (and in particular to $\overline{B(r)}$) of radius $\widehat r_0$ so that $2\pi \widehat r_0=\haus (\partial \widehat B)$. This with point (i) implies that $\widehat r_0>r_0$, thus the strict inclusion $B(r_0)\subset B(\widehat r_0)$ holds.
Recalling again the strict monotonicity of the first eigenvalue yields 
\begin{equation}\label{dim33}
\la_1(B(\widehat r_0)\setminus \overline{B(r)})< \la_1(B(r_0)\setminus \overline{B(r)}),
\end{equation}
which combined with \eqref{dim31} and \eqref{dim32} implies \eqref{ineq}.

\smallskip
\emph{Case 3: $K$ is neither a disk nor a segment and it is not contained in $B(r_0)$}. We use an approximation argument. 
For every $\delta>0$, we consider the disk $B(r_0+\delta)$ and notice that \eqref{inequality} still holds with $r_0$ replaced by $r_0+\delta$.  Since the set $K$ is contained in the open disk $B(r_0+\delta)$ we can follow the same strategy adopted in the previous \emph{Case 2} with the disk $B(r_0)$ replaced by $B(r_0+\delta)$ and consider a set $\widehat B_\delta$ satisfying the three items listed above. Without loss of generality, in the item (ii) we can also require that $\widehat B_\delta\supset B(r_0)$ so that, by items (ii) and (iii) $\widehat r_\delta\geq \widehat r_0$, where $\widehat r_\delta$  by definition satisfy $2\pi\widehat r_\delta=\haus(\partial \widehat B_\delta)$ and $\widehat r_0$ as in \emph{Case 2}. Therefore, similar inequalities to \eqref{dim31} and \eqref{dim32} yields
\[
\la_1(B(r_0+\delta)\setminus K)< \la_1({B(\widehat r_\delta)}\setminus \overline{B(r)})\leq  \la_1({B(\widehat r_0)}\setminus \overline{B(r)}).
\]
Letting $\delta\to 0$, from the Sverak continuity result \cite{sve} we obtain
\[
\la_1(B(r_0)\setminus K)\leq \la_1({B(\widehat r_0)}\setminus \overline{B(r)}),
\]
which combined with \eqref{dim33} implies the strict inequality \eqref{ineq}.

\smallskip

\emph{Case 4: K is a segment}.
We use again an approximation argument, but now on the set $K$. For a (small) real number $\delta>0$, consider 
the rectangle $K_\delta$ with a longest side onto $K$ of length $(1-\delta)L/2$ and smallest sides of length $\delta L/2$ to be contained in $B(r_0)$. By definition, $K_\delta$ belongs to \emph{Case 2} and $\haus(\partial K_\delta)=L$. Therefore, from \eqref{dim31} and \eqref{dim32} we have
\[
\la_1(B(r_0)\setminus K_\delta)< \la_1({B(\widehat r_\delta)}\setminus \overline{B(r)}),
\]
where $2\pi\widehat r_\delta=\haus(\partial \widehat B_\delta)$ for some $\widehat B_\delta$ satisfying the three items in \emph{Case 2} (in particular (iii) with $K_\delta$ in place of $K$).
Letting $\delta\to 0$, since $K_\delta\to ^H K$ and $\widehat r_\delta\to \widehat r$ for some $\widehat r>r_0$. Then by the Sverak continuity result and \eqref{dim33} we arrive to the inequality \eqref{ineq} also when $K$ is a segment.
\end{proof}

\begin{remark}
Theorem~\ref{casedisk} generalizes a result contained in \cite{hakrku}, about the maximization of the first Dirichlet eigenvalue of the Laplacian with circular shaped obstacles.
\end{remark}

\subsection{Annular domains}

We discuss the symmetry of a solution to \eqref{problem} in the case the domain $\Om$ is a ring. Due to topological reasons,  every maximizer of \eqref{problem} can not be radially symmetric, whenever $L$ is less than twice the perimeter of the inner disk (recall that by (iii) of Theorem~\ref{teo.reg} it cannot be a disk). Surprisingly this \emph{symmetry breaking} appears for other values of the perimeter constraints $L$, namely for those close to the perimeter of the inner disk. However, for large values of the constraint, namely for those close to the perimeter of the ring, the solution is provided by a set with \emph{full} symmetry. 

\begin{theorem}\label{casering}
Let $\Om=B(r_0)\setminus \overline{B(r_1)}$ be a ring, where $B(r_0)$ and $B(r_1)$ are concentric open disks of radii $0<r_1<r_0$. 
According to the value of the parameter $L\in (0,2\pi(r_0+r_1))$ the following properties hold. 

\begin{description}
\item[(i)] There exists $\alpha_0>0$ such that if $L>2\pi(r_0+ r_1)-\alpha_0$ then problem \eqref{problem}
has a unique solution: the maximizer $\K$ is given by the ring $\overline{B(r)}\setminus {B(r_1)}$ concentric to $\Om$ of radius $r=L/(2\pi)-r_1$.

\item[(ii)] There exists $\alpha_1>0$ such that if $L<4\pi r_1+\alpha_1$ then every maximizer $\K$ of problem \eqref{problem} is not radially symmetric.
\end{description}
\end{theorem}

\begin{proof}
(i) Let $\{L_n\}$ be a sequence of real numbers with $L_n\uparrow 2\pi (r_0+r_1)$ as $n\to \infty$. Consider a sequence of maximizers $\{K_n\}$ of problem \eqref{problem} associated to the length constraints $\{L_n\}$ such that $\mi(K_n)\to\haus(\partial\Om)$ and $\lambda_1(\Omega\setminus K_n)\uparrow \infty$ as $n\to\infty$, and in particular $K_n\to^H \overline\Om$.

Firstly we claim the existence of $n_0\in\mathbb N$ such that for every $n>n_0$ the set $K_n$ has two residual domains in $\mathbb R^2$, each of which contains one of $\overline{\Om}$. 
If not, by (iii) of Theorem~\ref{teo.opt}, $K_n$ has only one residual domain in $\mathbb R^2$ with $\partial K_n$ connected and therefore we may assume the existence of a curve $\gamma$ joining $\partial B(r_0)$ to $\partial B(r_1)$ with $\haus(\gamma)\geq r_0-r_1>0$ and $\haus(\gamma\cap \partial\Omega)=0$ so that $K_n\subset \overline{\Omega}\setminus \gamma$ for every $n$. This inclusion implies the convergence, up to subsequences (not relabelled), $\partial K_n\to^H J$ where $J$ is a continuum with $J\supset \partial \Om\cup \gamma$. The inclusion $J\supset \partial \Om$ is a consequence of $K_n\to^H \overline{\Om}$; while  $J\supset \gamma$ follows from the fact that for every $x\in \gamma$ there exists a sequence of points $\{x_n\}$ with $x_n\in \partial K_n$ (since $K_n$ is closed) such that $x_n\to x$ (for every $\epsilon>0$  if $n$ is large enough the set $\Om\setminus K_n$ can not contain the open enlargement $\{x\in\Omega:\, \mathrm{d}_\gamma(x)<\epsilon\}$, otherwise the eigenvalue would remain uniformly bounded). By the additivity and monotonicity properties of $\haus$ with the Go\l ab theorem and Corollary \ref{cor.residual} we have
\[
\haus(\partial \Om)+\haus(\gamma)\leq \haus(J)\leq  \liminf_n \haus(\partial K_n)\leq   \liminf_n \mi(K_n)=\haus(\partial \Om).
\]
This implies $\haus(\gamma)=0$ a contradiction and the claim at the beginning of the proof holds. 

Now we claim that, for every $n>n_0$ with $n_0$ given by the previous claim, the \emph{bounded} residual domain $\omega$ of $K_n$ in $\mathbb R^2$ is $B(r_1)$ and, moreover, that its outer boundary is the boundary of a convex set. If not, the set $\hull(K_n)\setminus B(r_1)$ would be a better competitor contradicting the optimality of $K_n$. Indeed, let $\omega_0$ be the \emph{unbounded} residual domain of $K_n$ in $\mathbb R^2$. Since $\haus(\partial \hull(K_n))<\haus(\partial\omega_0)$ and $\haus(\partial B(r_1))<\haus(\partial \omega)$, by Theorem~\ref{teo.reg} and Corollary~\ref{cor.residual} we have 
\[
\begin{split}
\mi(\hull(K_n)\setminus B(r_1))&\!=\!\haus(\partial (\hull(K_n)\setminus B(r_1)))\!=\haus(\partial \hull(K_n))+\haus(\partial B(r_1))
\\ &<\haus(\partial \omega_0)+\haus(\partial \omega)=\haus(\partial K_n) \leq \mi(K_n).
\end{split}
\]
  
The previous claims implies that for every $n>n_0$, we may look for the solution of problem \eqref{problem} only among sets $K_n$ enclosing the internal disk $B(r_1)$, satisfying $\mi(K_n)=\mi(K_n\cup B(r_1))+\mi(B(r_1))$. This allows to work with $K_n\cup B(r_1)$ as unknown, to apply Theorem~\ref{casedisk} to this set with perimeter constraint $L-2\pi r_1$ and to obtain the thesis on the radial symmetry.

(ii) By items (ii) and (iii) of Theorem~\ref{teo.opt} and (iii) of Theorem~\ref{teo.reg} if a maximizer is radially symmetric then it is necessarily a ring of perimeter $L$ containing the interior disk $B(r_1)$. Obviously, the symmetry breaking holds true whenever $L< 4\pi r_1$ (because there are no rings containing the interior disk). In the case $L=4\pi r_1$ the unique possibility for $\K$ to be radially symmetric would be $\partial B(r_1)$, but this is clearly not a maximizer. We call $K_1$ the maximizer corresponding to this perimeter constraint. Clearly $\lambda_1(\Omega\setminus K_1)>\lambda_1(\Omega)$.
Now, we assume $L>4\pi r_1$ and write $L=4\pi r_1 +2\pi \alpha$ for some $0< \alpha\leq r_0-r_1$. Moreover,
we choose the largest $\alpha_1$ such that
\begin{equation}\label{breaking}
\la_1\big(\Omega\setminus \overline{B(r_1+\alpha_1)} \big)\leq \la_1(\Omega\setminus K_1).
\end{equation}
By the Sverak continuity result \cite{sve} and the previous point (i) we deduce that $0<\alpha_1<r_0-r_1$.
Moreover, among all those rings $\overline{B({s})}\setminus {B(r)}$ with $r_1\leq r\leq r_0$ and $s:=L/(2\pi)-r$ so that the total perimeter is $L$, from the second claim in the previous proof of point (i), it follows that the ring $\overline{B(r_1+\alpha)}\setminus B(r_1)$ that is glued onto the interior disk $B(r_1)$ is the one which realizes the largest eigenvalue, namely
\[
\lambda_1\big(\Omega\setminus (\overline{B(s)}\setminus {B(r)})\big)\leq \lambda_1\big(\Omega\setminus (\overline{B(r_1+\alpha)}\setminus B(r_1))\big).
\]
This combined with \eqref{breaking} proves that any radially symmetric set cannot be optimal for problem \eqref{problem} in the case $L<4\pi r+\alpha_1$.
\end{proof}

\subsection{Perforated domains}

We can collect Theorem~\ref{casedisk} and Theorem~\ref{casering} to identify explicit solutions also for more general domains $\Omega$. 
\begin{theorem}%\label{casemore}
Let $k\in\mathbb N$, let $r_0>0$ and let $\{C_i\}_{i=1}^k$ be a family of $k$ open convex sets that are separated and strictly contained in $B(r_0)$, namely
\begin{equation}\label{separated}
d:=\min_{1\leq i< j\leq k+1}\min_{x\in \overline{C_i}, y\in \overline{C_j}} |x-y|>0,
\end{equation} 
where $C_{k+1}:=\partial B(r_0)$.
Let $\Om=B(r_0)\setminus \bigcup_{i=1}^k \overline{C_i}$ be an open disk of radius $r_0$ with $k$ convex holes. 
There exists $\alpha_{0}>0$ such that if $L>\haus(\partial \Om)-\alpha_{0}$ then problem \eqref{problem} has a unique solution: the maximizer $\K$ is given by the perforated disk $\overline{B(r)}\setminus  \bigcup_{i=1}^k C_i$ with $B(r)$ is the disk concentric to $B(r_0)$ with $r=L/(2\pi)-\sum_{i=1}^k\haus(\partial C_i)/(2\pi)$.
\end{theorem}
\begin{proof}
The proof follows the ones of Theorem~\ref{casedisk} and of Theorem~\ref{casering} . Therefore, we only sketch the principal steps.
Let $\{L_n\}$ be a sequence of real numbers with $L_n\uparrow \haus(\partial \Om)$ as $n\to \infty$. Consider a sequence of maximizers $\{K_n\}$ of problem \eqref{problem} associated to the length constraints $\{L_n\}$ so that $\lambda_1(\Omega\setminus K_n)\uparrow \infty$ as $n\to\infty$, and in particular $K_n\to^H \overline\Om$. For every $n$ large enough, the following facts hold.

- The set $K_n$ has \emph{exactly} $k+1$ residual domains in $\mathbb R^2$. Otherwise there would exist a curve $\gamma$ connecting two points belonging to two different connected components of $\partial \Om$ such that $\haus(\gamma)\geq d>0$ (recall the assumption \eqref{separated}) and by the Go\l ab theorem one would reach a contradiction.

- Every \emph{bounded} residual domain $\omega$ of $K_n$ in $\mathbb R^2$ \emph{coincides} with $C_i$ for $i=1,\dots, k$. This is a consequence of Theorem~\ref{teo.reg} and Corollary~\ref{cor.residual}, with the fact that all the sets $C_i$ are convex.

- One concludes by applying Theorem~\ref{casedisk} to the sets $\bigcup_{i=1}^k C_i\cup K_n$ with perimeter constraint $L-\sum_{i=1}^k\haus(\partial C_i)$.
\end{proof}

\end{document}